\documentclass[11pt]{amsart}
\usepackage{amsxtra,amssymb,amsmath,amsthm,amsbsy}
\usepackage[all]{xy}

\theoremstyle{plain}

\numberwithin{equation}{section}

\newtheorem{theorem}{Theorem}[section]
\newtheorem{lemma}[theorem]{Lemma}
\newtheorem{definition-lemma}[theorem]{Definition-Lemma}
\newtheorem{proposition}[theorem]{Proposition}

\newtheorem{example}[theorem]{Example}
\newtheorem{remark}[theorem]{Remark}

\newcommand{\id}         {{\mathrm {Id}}}

\newcommand{\Aut}        {{\mathrm {Aut}}}

\newcommand{\Ad}         {{\mathrm {Ad}}}

\newcommand{\pr}         {{\mathrm{pr}}}
\newcommand{\Diff}       {\mathrm{Diff}}

\newcommand{\Der}        {\mathrm{Der}}
\newcommand{\Sym}        {\mathrm{Sym}}
\newcommand{\sym}        {\mathrm{sym}}

\newcommand{\Grd}        {\mathcal{G}}

\newcommand{\SP} [1]     {{\left\langle {{#1}} \right\rangle}}

\newcommand{\frakg}     {\mathfrak{g}}

\newcommand{\sour}        {\mathsf{s}}
\newcommand{\tar}         {{\mathsf{t}}}

\newcommand{\Cour}[1]      {[\![#1]\!]}

\newcommand{\Lie}        {\mathcal L}

\newcommand{\dd}     {\mathrm{d}}

\begin{document}
\title[]
{Symmetries and reduction of multiplicative 2-forms}

\author[]{Henrique Bursztyn and  Alejandro Cabrera}

\address{Instituto de Matem\'atica Pura e Aplicada,
Estrada Dona Castorina 110, Rio de Janeiro, 22460-320, Brasil }
\email{henrique@impa.br, cabrera@impa.br}

\address{Departamento de Matem\'atica Aplicada,
Instituto de Matem\'atica, Universidade Federal do Rio de Janeiro,
CEP 21941-909, Rio de Janeiro - RJ, Brazil} \email{cabrera@impa.br}

\date{}

\maketitle

\begin{abstract}
This paper is concerned with symmetries of closed multiplicative
2-forms on Lie groupoids and their infinitesimal counterparts. We
use them to study Lie group actions on Dirac manifolds by Dirac
diffeomorphisms and their lifts to presymplectic groupoids, building
on recent work of Fernandes-Ortega-Ratiu \cite{FOR} on Poisson
actions.
\end{abstract}

\begin{center}
{\it Dedicated to Tudor Ratiu on the occasion of his 60th birthday}
\end{center}

\tableofcontents

\section{Introduction}\label{sec:intro}

This paper studies symmetries of closed multiplicative 2-forms on
Lie groupoids as well as their infinitesimal counterparts. This
study leads to an extension of some of the recent work of
Fernandes-Ortega-Ratiu \cite{FOR} on Poisson actions and their
Hamiltonian lifts to symplectic groupoids to the framework of Dirac
structures.

Consider a Lie group $G$ acting on a Poisson manifold $M$ by Poisson
diffeomorphisms, and suppose that $\Grd$ is the
source-simply-connected symplectic groupoid integrating $M$ (see
e.g. \cite{CDW}). A key observation that can be traced back to
\cite{MiWe} is that, even when the $G$-action on $M$ does not admit
a momentum map, it can always be lifted to a Hamiltonian $G$-action
on $\Grd$ with a momentum map $J: \Grd\to \frakg^*$ suitably
compatible with the groupoid structure (in the sense of
\eqref{eq:Jcocycle} below).

In \cite{FOR}, Fernandes, Ortega and Ratiu use this canonical
momentum map on $\Grd$ as a tool for studying actions on $M$ by
Poisson diffeomorphisms. For example, when the action is free and
proper, so that the quotient $M/G$ is again a Poisson manifold, they
show that the Marsden-Weinstein quotient $J^{-1}(0)/G$ is a
symplectic groupoid for $M/G$ (though not necessarily the
source-simply-connected one). The simplest instance of this result
is when $M$ is an arbitrary manifold (equipped with the zero Poisson
structure), in which case one just recovers the well-known fact that
any $G$-action on $M$, when lifted to $T^*M$, is Hamiltonian with
momentum map 
\begin{equation}\label{eq:Jcan0}
J_{can}:T^*M \to \frak{g}^*,\;\;\;
\SP{J_{can}(\alpha),u}=\alpha(u_M),
\end{equation}
where $u\in \frak{g}$ and $u_M$ is the infinitesimal generator for
the $G$-action on $M$. The reduction of $T^*M$ at level zero is
$T^*(M/G)$, which is just the symplectic groupoid of $M/G$. Our goal
in this note is to place some of these results from \cite{FOR} in
the broader context of Dirac structures \cite{courant}, offering an
alternative viewpoint based on methods from \cite{BuCa,BCO} and
pointing out additional subtleties that arise in this generality.

Our starting point is \cite{bcwz}, where the global objects
associated with Dirac manifolds are identified; these objects are
referred to as \textit{presymplectic groupoids}, and their role in
the theory of Dirac structures is entirely analogous to that of
symplectic groupoids in Poisson geometry. A presymplectic groupoid
is a Lie groupoid equipped with a closed multiplicative 2-form but,
in contrast with symplectic groupoids, this 2-form may be degenerate
(though in a ``controlled'' way). Let $(\Grd,\omega)$ be a
source-simply-connected presymplectic groupoid integrating a Dirac
structure $L$ on $M$, and suppose that a Lie group $G$ acts on $M$
by Dirac diffeomorphisms. We observe that, just as in the case of
Poisson actions, the action on $M$ lifts to a Hamiltonian action on
$\Grd$, in the sense that there is a $G$-equivariant map $J:\Grd\to
\frakg^*$ satisfying
\begin{equation}\label{eq:mmapcond}
i_{u_{\Grd}}\omega = -\dd \SP{J,u}, \;\;\; \forall u \in \frak{g},
\end{equation}
where $u_{\Grd}$ is the infinitesimal generator associated with $u$
for the $G$-action on $\Grd$. Note, however, that condition
\eqref{eq:mmapcond} does not necessarily determine the infinitesimal
action since $\omega$ may be degenerate.

When the $G$-action on $M$ is free and proper, the quotient $M/G$ is
a smooth manifold that inherits (with the aid of extra regularity
conditions to be specified in Section~\ref{sec:dirac}) natural
geometrical structures: on the one hand, a Marsden-Weinstein type
reduction of $\Grd$ produces a Lie groupoid $\Grd_{red}$ over $M/G$
equipped with a closed multiplicative 2-form; on the other hand, the
$G$-invariant Dirac structure $L$ on $M$ may be pushed forward to a
Dirac structure $L_{quot}$ on $M/G$. Unlike the case of Poisson
manifolds, however, $\Grd_{red}$ is generally \textit{not} a
presymplectic groupoid for $(M/G,L_{red})$ (in the terminology of
\cite[Sec.~4]{bcwz}, $\Grd_{red}$ will be proven to be an
\textit{over}-presymplectic groupoid). As we will discuss, it may
happen that the quotient $(M/G,L_{quot})$ does not admit any
presymplectic groupoid at all; i.e., $L_{quot}$ may not be
integrable as a Lie algebroid, despite the fact that $L$ was. Our
approach to Dirac structures and presymplectic groupoids is through
the notion of \textit{IM 2-forms} \cite{bcwz} (see also
\cite{AC,BuCa,BCO}), differently from \cite{FOR} that uses path
spaces \cite{CF} (see also \cite{catfel,severa}). At the end of the
paper, we discuss how to reconcile both viewpoints.

The increasing role of Dirac structures in mechanical problems has
been one of our motivations to revisit the work in \cite{FOR} from
this broader perspective. Even in the context of Poisson geometry,
Dirac structures naturally appear in the stratified geometry of
non-free Poisson actions, see Sections~2.3 and 4.3 of \cite{FOR}. A
more complete picture, that we leave for future work, should include
proper actions on Dirac manifolds which are not necessarily free;
Jotz, Ratiu and Sniatycki have begun this study in \cite{JR,JRS}.
Another source of inspiration for the present note is the
possibility of using Dirac structures as a tool for studying
symplectic groupoids of Poisson homogeneous spaces \cite{Lu}, yet to
be explored.

The paper is structured as follows: in Section \ref{sec:setup}, we
recall IM 2-forms on Lie algebroids and present the geometric set-up
that will be considered in the paper; Section \ref{sec:momentum}
discusses the Hamiltonian properties of actions on Lie groupoids
that lift symmetries of closed IM 2-forms; Section
\ref{sec:infreduction} describes the reduction of closed IM 2-forms,
while Section \ref{sec:globalred} deals with its global version,
i.e., reduction of closed multiplicative 2-forms; the particular
situation of Dirac structures is treated in Section \ref{sec:dirac};
in Section \ref{sec:pathspaces}, we relate the approach in this
paper to the viewpoint of path spaces in \cite{FOR}.

It is a great pleasure to dedicate this note to Tudor, whose work
has been central in so many aspects of Poisson geometry and
geometric mechanics, including many of the issues discussed here. We
are thankful to him for his mathematical insights, unmatched
enthusiasm, and constant support.

\subsection*{Acknowledgments} We thank Rui L. Fernandes for helpful advice
and Juan-Pablo Ortega for encouraging the writing of this note. We
also thank the referees for valuable comments. Bursztyn's research
was supported by CNPq and Faperj. Cabrera thanks the University of
Toronto and IMPA for their hospitality while this work was being
completed and CNPq for financial support.


\section{Preliminaries}\label{sec:setup}

\subsection{Symmetries of vector bundles}\label{subsec:symm}
Let $M$ be a smooth manifold. For a smooth vector bundle $E\to M$,
$\Aut(E)$ denotes its group of vector-bundle automorphisms. An
automorphism of $E$ covers a diffeomorphism of $M$, so there is a
homomorphism $\Aut(E)\to \Diff(M)$. Any $\Phi\in \Aut(E)$, covering
$\varphi\in \Diff(M)$, acts on the space of sections of $E$ by
\begin{equation}\label{eq:Psi*}
\Phi_*: \Gamma(E)\to \Gamma(E),\;\; \Phi_*(s) = \Phi\circ s\circ
\varphi^{-1}.
\end{equation}

Let $G$ be a Lie group with Lie algebra $\frakg$. In this paper, an
action of $G$ on $E$ is always assumed to be by vector-bundle
automorphisms, i.e., defined by a group homomorphism $G\to \Aut(E)$;
any $G$-action on $E$ naturally covers, and it is said to be a
\textit{lift} of, an action of $G$ on $M$.

The space of infinitesimal automorphisms of $E$, denoted by
$\Der(E)$, is the set of pairs $(X,D)$, where $X\in \mathfrak{X}(M)$
is a vector field on $M$ and $D: \Gamma(E)\to \Gamma(E)$ is a linear
endomorphism such that
\begin{equation}\label{eq:Deq}
D(fs)=(\Lie_Xf)s + fD(s), \;\;\; \forall f\in C^\infty(M),\, s\in
\Gamma(E).
\end{equation}
For a one-parameter family of automorphisms $\Phi_t$ in $\Aut(E)$,
$\Phi_0=Id$, the corresponding infinitesimal automorphism is given
by $D(s) = \frac{d}{dt}\big |_{t=0} (\Phi_t)_*(s)$. The space
$\Der(E)$ has a natural Lie algebra structure, given by commutators,
for which the projection $\Der(E)\to \mathfrak{X}(M)$, $(X,D)\mapsto
X$, is a Lie algebra homomorphism. Any $G$-action on $E$ gives rise
to a Lie algebra homomorphism $\frakg\to \Der(E)$ for which the
composition $\frakg\to \Der(E)\to \mathfrak{X}(M)$ agrees with the
infinitesimal counterpart of the $G$-action on $M$ covered by the
$G$-action on $E$.

Suppose $G$ acts on $M$ freely and properly; let $p:M\to M/G$ be the
natural projection, which is a surjective submersion. Any lift of
this action to a $G$-action on $E$ determines a unique vector bundle
(up to isomorphism) over $M/G$, denoted by
\begin{equation}\label{eq:EG}
E/G \to M/G,
\end{equation}
such that the pull-back bundle $p^*(E/G)$ is naturally identified
with $E$, and the pull-back of sections identifies $\Gamma(E/G)$
with the space of $G$-invariant sections $\Gamma(E)^G$.

\subsection{The geometric set-up} \label{subsec:setup}
This paper will be concerned with the following geometric objects:

\begin{itemize}

\item[$(i)$] A smooth manifold $M$ acted upon by a Lie group $G$.

\item[$(ii)$] A Lie algebroid $A\to M$, with Lie bracket $[\cdot,\cdot]$ on
$\Gamma(A)$ and anchor map $\rho:A\to TM$, equipped with a closed IM
2-form \cite{bcwz} $\mu:A\to T^*M$.

\item[$(iii)$] An action of $G$ on $A$ by Lie algebroid automorphisms which lifts the
$G$-action on $M$ and for which $\mu:A\to T^*M$ is $G$-equivariant.

\end{itemize}

Let us briefly recall the objects that appear in $(ii)$ and $(iii)$.

The notion of \textit{closed IM 2-form} \cite{bcwz} in $(ii)$ refers
to a morphism of vector bundles $\mu: A\to T^*M$, covering the
identity map, satisfying
\begin{align}\label{eq:IM1}
&i_{\rho(a)}\mu(b)=-i_{\rho(b)}\mu(a)\\
&\mu([a,b])=\Lie_{\rho(a)}\mu(b)-i_{\rho(b)}d\mu(a) =
\Lie_{\rho(a)}\mu(b)-\Lie_{\rho(b)}\mu(a) + \dd
i_{\rho(b)}\mu(a).\label{eq:IM2}
\end{align}

In $(iii)$, we consider the group of \textit{automorphisms} of the
Lie algebroid $A$, i.e., the subgroup of vector-bundle automorphisms
$\Aut(A)$ defined by
\begin{equation}\label{eq:SYM}
\Sym(A)=\{\Phi:A\to A \in \Aut(A)\,|\,
\Phi_*([a,b])=[\Phi_*(a),\Phi_*(b)] \, \forall\, a,b\in \Gamma(A)\}.
\end{equation}
One can directly verify that any $\Phi \in \Sym(A)$, covering
$\varphi \in \Diff(M)$, satisfies
\begin{equation}\label{eq:anchor}
\rho \circ \Phi = \dd\varphi \circ \rho,
\end{equation}
as a result of the identity $\Phi_*([a,fb])=[\Phi_*(a),\Phi_*(fb)]$,
$\forall f\in C^\infty(M)$, and the Leibniz rule. The infinitesimal
Lie-algebroid automorphisms form a Lie subalgebra of $\Der(A)$ given
by
\begin{equation}\label{eq:sym}
\sym(A)=\{(X,D)\in \Der(A)\,|\, D([a,b])=[D(a),b] + [a,D(b)],
\forall a,b \in \Gamma(A)\}.
\end{equation}
In analogy with \eqref{eq:anchor}, any $(X,D)\in \sym(A)$ satisfies
\begin{equation}\label{eq:anchor2}
\rho(D(a)) = [X,\rho(a)], \;\;\; \forall a \in \Gamma(A);
\end{equation}
this can be verified directly via the identity $D([a,fb])=[D(a),fb]+
[a,D(fb)]$, $\forall f\in C^\infty(M)$, the Leibniz identity and
\eqref{eq:Deq}.

The $G$-action on $A$ in $(iii)$ is given by a group homomorphism
$$
G\to \Sym(A)\subseteq \Aut(A).
$$
The equivariance of $\mu$ in $(iii)$ is with respect to the
canonical lift of the $G$-action on $M$ to $T^*M$. The $G$-action on
$A$ gives rise to an infinitesimal $\frakg$-action on $A$, defined
by a Lie algebra homomorphism $\frakg\to \sym(A) \subseteq \Der(A)$
necessarily of the form
\begin{equation}\label{eq:infaction}
u \mapsto (u_M, D),
\end{equation}
where $u_M$ is the infinitesimal generator of the $G$-action on $M$
defined by $u\in \frakg$. Infinitesimally, the equivariance of $\mu$
becomes
\begin{equation}\label{eq:equiv}
\Lie_{u_M}\mu(a) = \mu (D(a)),\;\;\; \forall a\in \Gamma(A),
\end{equation}
for all $(u_M,D)\in \sym(A)$ as in \eqref{eq:infaction}.

A central theme in this paper is transferring the geometric
information in $(i)$, $(ii)$, $(iii)$ to a Lie groupoid $\Grd$
integrating $A$. Our notation and conventions for a Lie groupoid
$\Grd$ over $M$ are as follows: the source and target maps are
denoted by $\sour$ and $\tar$, the set $\Grd^{(2)}\subset \Grd\times
\Grd$ of composable pairs $(g,h)$ is defined by the condition
$\sour(g)=\tar(h)$, and the multiplication is denoted by
$m:\Grd^{(2)}\to \Grd$, $m(g,h)=gh$; the unit map $M \hookrightarrow
\Grd$ is used to identify $M$ with its image in $\Grd$. The Lie
algebroid of $\Grd$ is $A\Grd=\ker(\dd \sour)|_M$, with anchor
$\dd\tar|_A:A\to M$ and bracket induced by right-invariant vector
fields.

\subsection{Dirac structures} \label{subsec:dirac}
The main example of the set-up $(i)$, $(ii)$, $(iii)$ that we will
have in mind is given as follows. Consider the vector bundle
$\mathbb{T}M:= TM\oplus T^*M$ over $M$ equipped with the
nondegenerate symmetric fibrewise bilinear form given at each $x\in
M$ by
\begin{equation}\label{eq:pair}
\SP{(X,\alpha),(Y,\beta)}=\beta(X)+ \alpha(Y),\;\; X,Y \in T_xM,\; \alpha,\beta \in T_x^*M,
\end{equation}
and the Courant bracket $\Cour{\cdot,\cdot}:\Gamma(\mathbb{T}M)\times \Gamma(\mathbb{T}M) \to
\Gamma(\mathbb{T}M)$,
\begin{equation}\label{eq:cour}
 \Cour{(X,\alpha),(Y,\beta)}=([X,Y],\Lie_X\beta - i_Y\dd\alpha).
\end{equation}
We denote by $\pr_{T}: \mathbb{T}M\to TM$ and
$\pr_{T^*}:\mathbb{T}M\to T^*M$ the canonical projections.

A \textit{Dirac structure} \cite{courant} on $M$ is a vector
subbundle $L\subset \mathbb{T}M$ satisfying $L=L^\perp$ with respect
to $\SP{\cdot,\cdot}$, and which is involutive with respect to
$\Cour{\cdot,\cdot}$, i.e., $\Cour{\Gamma(L),\Gamma(L)}\subseteq
\Gamma(L)$. Any Poisson structure $\pi \in \Gamma(\wedge^2TM)$ may
be viewed as a Dirac structure via its graph, i.e.,
\begin{equation*}\label{eq:Lpi}
L=\{(\pi^\sharp(\alpha),\alpha)\,|\, \alpha\in T^*M\},
\end{equation*}
where $\pi^\sharp: T^*M\to TM$ is given by
$\pi^\sharp(\alpha)=i_\alpha\pi$; closed 2-forms may be viewed as
Dirac structures analogously, as graphs of the associated bundle
maps $TM\to T^*M$.

For a Dirac structure $L$ on $M$, the vector bundle $L\to M$
inherits a Lie algebroid structure, with bracket on $\Gamma(L)$
given by the restriction of $\Cour{\cdot,\cdot}$  and anchor given
by restriction of $\pr_T$ to $L$:
\begin{equation}\label{eq:LAstr}
[\cdot,\cdot]_L = \Cour{\cdot,\cdot}|_{\Gamma(L)},\;\;\; \rho_L :=
\pr_T|_L: L\to TM.
\end{equation}
Moreover, the map
\begin{equation}\label{eq:diracIM}
\mu_L:= \pr_{T^*} |_L : L \to T^*M
\end{equation}
is a closed IM 2-form\footnote{Conversely, given any Lie algebroid
$A$ and closed IM 2-form $\mu:A\to T^*M$, we may consider the image
of the map $A\to \mathbb{T}M$, $a\mapsto (\rho(a),\mu(a))$; if
$\mathrm{rank}(A)=\dim(M)$ and $\ker(\rho)\cap \ker(\mu)=\{0\}$,
then this map is an isomorphism from $A$ onto its image, which is
then a Dirac structure on $M$.}.

A diffeomorphism $\varphi:M\to M$ is said to preserve a Dirac
structure $L$ if its lift
\begin{equation}\label{eq:lift}
 (\dd\varphi,(\dd\varphi^{-1})^*):\mathbb{T}M\to \mathbb{T}M
\end{equation}
preserves $L$; equivalently, the following holds for all $x\in M$:
$$
L_{\varphi(x)}=\{ (\dd\varphi(X),(\dd\varphi^{-1})^*(\alpha))\,|\,
(X,\alpha) \in L_x \}.
$$
In particular, $\varphi$ is both a backward and forward Dirac map
(see e.g. \cite[Sec.~2.1]{bcwz} and references therein).

The situation that will concern us is that of a  Lie group $G$
acting, freely and properly, on a manifold $M$ by diffeomorphisms
preserving a Dirac structure $L\subset \mathbb{T}M$. This fits into
the framework of $(i)$, $(ii)$, $(iii)$ if we set $A=L$, $\mu=\mu_L$
as in \eqref{eq:diracIM}, and the $G$-action on $A$ to be the
natural lift of the $G$-action on $M$ via \eqref{eq:lift} restricted
to $L$. When $M$ is a Poisson manifold, the Lie algebroid in
question is naturally identified with $T^*M$, in such a way that
$\mu_L$ is just the identity map; in this case, the information in
$(i)$, $(ii)$, $(iii)$ boils down to an action on $M$ by Poisson
diffeomorphisms, as considered in \cite{FOR}.

Although we are primarily interested in Dirac structures, it turns
out that more general closed IM 2-forms naturally arise when one
considers reduction by symmetries; so we will work in the more
general context of $(i)$, $(ii)$, $(iii)$ from the outset.

\section{Momentum maps and lifted Hamiltonian actions}
\label{sec:momentum}

Let us consider the geometric set-up described in $(i)$, $(ii)$,
$(iii)$ of Section \ref{subsec:setup}. Let $J_{can}: T^*M \to
\frakg^*$ be the momentum map of the canonical lifting of the
$G$-action on $M$ to $T^*M$, given by \eqref{eq:Jcan0}, and let
\begin{equation}\label{eq:JA}
J_A :=J_{can}\circ \mu: A\to \frakg^*.
\end{equation}

\begin{lemma}\label{lem:morphism}
The map $J_A$ is a morphism of Lie algebroids, where $\frakg^*$ is
viewed as a trivial Lie algebroid over a point (i.e., a Lie algebra
with trivial bracket).
\end{lemma}

\begin{proof}
Note that $J_A: A\to \frakg^*$ is a morphism of vector bundles,
because both $\mu:A\to T^*M$ and $J_{can}:T^*M \to \frakg^*$ are.
The remaining condition to be checked is
\begin{equation}\label{eq:Jmorphism}
J_A([a,b])=\Lie_{\rho(a)}J_A(b)-\Lie_{\rho(b)}J_A(a), \;\;\;
\forall\, a,b \in \Gamma(A),
\end{equation}
as an equality in $C^\infty(M,\frakg^*)$.

Pairing $J_A([a,b])$ with $v\in \frakg$, and using \eqref{eq:Jcan0}
and \eqref{eq:IM2}, we get
$$
\SP{J_{can}(\mu([a,b])),v}=\SP{\mu([a,b]),v_M} =
\SP{\Lie_{\rho(a)}\mu(b)-\Lie_{\rho(b)}\mu(a) + \dd
i_{\rho(b)}\mu(a), v_M}.
$$
The right-hand-side of the previous equation agrees with
\begin{equation}\label{eq:rhs}
\SP{\Lie_{\rho(a)}\mu(b),v_M} - \SP{\Lie_{\rho(b)}\mu(a),v_M} +
\Lie_{v_M} i_{\rho(b)}\mu(a).
\end{equation}
Using \eqref{eq:IM1}, \eqref{eq:anchor2}, as well as the
infinitesimal equivariance of $\mu$ \eqref{eq:equiv}, we obtain
\begin{align*}
\Lie_{v_M} i_{\rho(b)}\mu(a) = &\SP{\Lie_{v_M}\mu(a),\rho(b)} +
\SP{\mu(a),[v_M,\rho(b)]}\\
= & \SP{\mu(D(a)),\rho(b)} + \SP{\mu(a),[v_M,\rho(b)]}\\
=& -\SP{\mu(b),\rho(D(a))} + \SP{\mu(a),[v_M,\rho(b)]}\\
=& -\SP{\mu(b),[v_M,\rho(a)]} + \SP{\mu(a),[v_M,\rho(b)]}.
\end{align*}
So \eqref{eq:rhs} becomes
\begin{equation}\label{eq:rhs2}
\SP{\Lie_{\rho(a)}\mu(b),v_M}  - \SP{\Lie_{\rho(b)}\mu(a),v_M}
-\SP{\mu(b),[v_M,\rho(a)]} + \SP{\mu(a),[v_M,\rho(b)]}.
\end{equation}
On the other hand
$$
\SP{\Lie_{\rho(a)}J_A(b),v}= \Lie_{\rho(a)}\SP{J_A(b),v}=
\SP{\Lie_{\rho(a)}\mu(b),v_M}+\SP{\mu(b),[\rho(a),v_M]},
$$
and swapping $a$ and $b$ gives
$$
\SP{\Lie_{\rho(b)}J_A(a),v}=
\SP{\Lie_{\rho(b)}\mu(a),v_M}+\SP{\mu(a),[\rho(b),v_M]}.
$$
Comparing with \eqref{eq:rhs2}, we see that \eqref{eq:rhs} equals
$$
\SP{\Lie_{\rho(a)}J_A(b) - \Lie_{\rho(b)}J_A(a),v},
$$
which immediately implies that \eqref{eq:Jmorphism} holds.
\end{proof}

Suppose that $A$ is integrable and $\Grd$ is the
source-simply-connected Lie groupoid integrating it. We now make use
of the Lie algebroid/groupoid version of Lie's second theorem (see
e.g. the appendix of \cite{Mac-Xu2}). Since the $G$-action on $A$ is
by Lie-algebroid automorphisms, it has a natural lift to a
$G$-action on $\Grd$ by Lie-groupoid automorphisms. Similarly, since
$J_A: A\to \frakg^*$ is a Lie algebroid morphism, it integrates to a
Lie groupoid morphism
\begin{equation}\label{eq:J}
J: \Grd\to \frakg^*,
\end{equation}
where $\frak{g}^*$ is now viewed as a groupoid over a point (i.e.,
an abelian group with respect to addition); i.e., $J$ satisfies
\begin{equation}\label{eq:Jcocycle}
J(gh)=J(g)+ J(h),
\end{equation}
for $(g,h)\in \Grd^{(2)}$. We also know from \cite{bcwz} (c.f.
\cite{AC,BuCa,BCO}) that the closed IM 2-form $\mu$ uniquely
integrates to a closed 2-form $\omega$ in $\Grd$ which is
\textit{multiplicative}, in the sense that
\begin{equation}\label{eq:mult}
m^*\omega=\pr_1^*\omega + \pr_2^*\omega,
\end{equation}
where $\pr_1,\pr_2: \Grd\times \Grd \to \Grd$ are the natural
projections and $m:\Grd^{(2)}\to \Grd$ is the multiplication on
$\Grd$. The relationship between $\omega$ and $\mu$ is given by
\begin{equation}\label{eq:muomega}
i_X \mu(a)= \omega(a, X)
\end{equation}
for $a\in A$, $X\in TM$, where we view $A$ and $TM$ as subbundles of
$T\Grd|_M$.

Note that the notion of being multiplicative, defined by condition
\eqref{eq:mult}, makes sense for differential forms of any degree;
in particular, \eqref{eq:Jcocycle} says that $J$ is a multiplicative
($\frak{g}^*$-valued) function.

\begin{proposition}\label{prop:lifted} The map $J: \Grd\to \frak{g}^*$ is
$G$-equivariant and satisfies
\begin{equation}\label{eq:mommap}
i_{v_\Grd}\omega = -\dd\SP{J,v}, \;\;\; \forall\, v\in \frak{g},
\end{equation}
where $v_\Grd$ is the infinitesimal generator defined by $v$ for the
$G$-action on $\Grd$.
\end{proposition}


\begin{proof}
Let us verify the $G$-equivariance of $J:\Grd\to \frak{g}^*$. Each
$\sigma \in G$ defines an automorphism $\phi_\sigma$ of $A$ as well
as its global counterpart $\Phi_\sigma$, which is an automorphism of
$\Grd$. Then $(\Ad^*_{\sigma})^{-1}\circ J_A \circ \phi_\sigma:A\to
\frak{g}^*$ is a Lie-algebroid morphism, whose global counterpart is
$(\Ad^*_{\sigma})^{-1}\circ J \circ \Phi_\sigma: \Grd\to
\frak{g}^*$. Since both $J_{can}$ and $\mu$ are $G$-equivariant, so
is $J_A$; i.e., $J_A=(\Ad^*_{\sigma})^{-1}\circ J_A \circ
\phi_\sigma$. It follows from the uniqueness of integration of
morphisms that $J= (\Ad^*_{\sigma})^{-1}\circ J \circ \Phi_\sigma$,
i.e, $J$ is $G$-equivariant. 

The fact that $G$ acts on $\Grd$ by Lie-groupoid automorphisms
implies that the infinitesimal generators of this action are
multiplicative vector fields (i.e., each $v_\Grd: \Grd\to T\Grd$ is
a groupoid morphism, see e.g. \cite{Mac-Xu3}). It follows that the
1-form $i_{v_\Grd}\omega$ is multiplicative. Since $J$ is also
multiplicative (see \eqref{eq:Jcocycle}), so is $-\dd\SP{J,v}$. To
show that they are the same, it suffices (see e.g. \cite{AC,BuCa})
to check that these 1-forms agree on elements $a\in A_x =
\ker(d\sour)|_x \subseteq T_x\Grd$, $x\in M$. Note that
$v_\Grd(x)=v_M(x)$ for all $x\in M$. So
$$
\omega_x(v_\Grd,a)=-\omega_x(a,v_M)=-i_{v_M}\mu(a).
$$
On the other hand, using that $J_A = \dd J|_A$,
$$
-\SP{\dd J(a),v}=- \SP{J_A(a),v}=-\SP{\mu(a),v_M}.
$$
\end{proof}

As in \cite[Thm.~3.3(ii)]{FOR}, one can check that there exists a
map $j:M\to \frakg^*$ such that $J= \sour^*j - \tar^*j$ if and only
if $\SP{\mu(a),u_M}= -\dd j^u(\rho(a))$, for all $u\in \frakg$,
$a\in A$, where $j^u(x)=\SP{j(x),u}$. The map $j$ should be seen as
a momentum map for the $G$-action on $M$ preserving the closed IM
2-form $\mu:A \to T^*M$.

\section{Reduction of closed IM
2-forms}\label{sec:infreduction}

We keep considering the set-up in $(i)$, $(ii)$, $(iii)$ of Section
\ref{subsec:setup}, and  $J_A: A\to \frakg^*$  as in \eqref{eq:JA}.
Let us denote by $K\subseteq TM$ the distribution tangent to the
$G$-orbits on $M$, and by $K^\circ \subseteq T^*M$ its annihilator.

\begin{lemma}\label{lem:JA}
If $J_A^{-1}(0)$ is a subbundle of $A$, then it is a Lie
subalgebroid of $A$.
\end{lemma}

\begin{proof}
The result follows from $\Gamma(J_A^{-1}(0))$ being closed under the
Lie bracket on $\Gamma(A)$, as a consequence of
\eqref{eq:Jmorphism}.
\end{proof}

\begin{remark} If $\mu$ has constant rank, then $J_A^{-1}(0)\subseteq A$ is a subbundle
if and only if $\mathrm{Im}(\mu) \cap K^\circ$ has constant rank.
This follows from $J_{can}^{-1}(0) = K^\circ$ and $J_A=J_{can}\circ
\mu$.
\end{remark}

Let us assume that $J_A^{-1}(0)$ is a subbundle of $A$ and that the
$G$-action on $M$ is free and proper. The $G$-equivariance of $J_A$
 implies that the $G$-action on $A$ keeps $J_A^{-1}(0)$ invariant;
so we may consider the vector bundle (see \eqref{eq:EG})
$$
A_{red}:= J_A^{-1}(0)/G \to M/G.
$$
The quotient map
\begin{equation}\label{eq:quotmap}
p:M\to M/G
\end{equation}
induces an isomorphism $p^*:
C^\infty(M/G)\to C^\infty(M)^G$, where $C^\infty(M)^G$ denotes the
space of $G$-invariant functions on $M$; also, $p^* A_{red} =
J_A^{-1}(0)$ and the pull-back of sections
$$
p^*: \Gamma(A_{red})\to \Gamma(J_A^{-1}(0))
$$
gives an identification between $\Gamma(A_{red})$ and
$\Gamma(J_A^{-1}(0))^G$.

\begin{proposition}\label{prop:Ared}
The vector bundle $A_{red} \to M/G$ inherits a natural Lie-algebroid
structure, and the closed IM 2-form $\mu$ on $A$ induces a closed IM
2-form $\mu_{red}$ on $A_{red}$.
\end{proposition}

\begin{proof}
Since the $G$-action on $J_A^{-1}(0)$ is by Lie-algebroid
automorphisms, the space $\Gamma(J_A^{-1}(0))^G$ is closed under the
Lie bracket, and the identification
$$p^*:
\Gamma(A_{red})\stackrel{\sim}{\to} \Gamma(J_A^{-1}(0))^G
$$
defines a Lie bracket $[\cdot,\cdot]_{red}$ on $\Gamma(A_{red})$ by
$ p^* [a,b]_{red} = [p^*a,p^*b]$. We define an anchor $\rho_{red}:
A_{red}\to T(M/G)$ by the condition
\begin{equation}\label{eq:rhored}
\rho_{red}(a) = \dd p(\rho(p^*a)),
\end{equation}
for $a\in \Gamma(A_{red})$; note that this expression is
well-defined (by the compatibility between $\rho$ and the
$G$-action, see \eqref{eq:anchor}) and $C^\infty(M/G)$-linear. For
$a$, $b \in \Gamma(A_{red})$ and $f\in C^\infty(M/G)$, the Leibniz
identity follows from
\begin{align*}
p^* ([a,fb]_{red})&=[p^*a, (p^*f) p^*b]=(\Lie_{\rho(p^*a)}p^*f)p^*b
+
p^*f [p^*a, p^*b]\\
&= p^* ( (\Lie_{\dd p(\rho(p^*a))}f) b + f[a,b]_{red})= p^* (
(\Lie_{\rho_{red}(a)}f) b + f[a,b]_{red}).
\end{align*}

Since $J_A=J_{can}\circ \mu$, it is clear that $\mu:A \to T^*M$
restricts to $J_A^{-1}(0)\to J_{can}^{-1}(0)= K^\circ$; since it is
$G$-equivariant, it descends to a bundle map
$$
\mu_{red}: A_{red}\to T^*(M/G),
$$
where we used the natural identification $K^\circ/G \cong T^*(M/G)$.
It is a direct verification that the conditions \eqref{eq:IM1},
\eqref{eq:IM2} for $\mu$ imply the same conditions for $\mu_{red}$.
\end{proof}

\section{Global reduction}\label{sec:globalred}

We now discuss the global reduction associated with the data in
$(i)$, $(ii)$, $(iii)$ of Section \ref{subsec:setup}. Let us assume
that $J_A^{-1}(0)\subseteq A$ has constant rank, so it is a Lie
subalgebroid of $A$, see Lemma~\ref{lem:JA}. We assume that $A$ is
an integrable Lie algebroid and that $\Grd$ is the
source-simply-connected Lie groupoid integrating it; let $J:\Grd\to
\frakg^*$ be the multiplicative function integrating $J_A$, as in
Section \ref{sec:momentum}. We say that $0 \in \frakg^*$ is a
\textit{clean value} for $J$ if $J^{-1}(0)$ is a submanifold and
$\ker(\dd J)|_g = T_g J^{-1}(0)$ for all $g\in J^{-1}(0)$.

\begin{lemma}\label{lem:subm} If $0$ is a clean value for $J$,
then $J^{-1}(0)$ is a Lie subgroupoid of $\Grd$ over $M$ whose Lie
algebroid is $J_A^{-1}(0)$.
\end{lemma}

\begin{proof}
Let $\tilde{\sour}, \tilde{\tar}: J^{-1}(0)\to M$ be the
restrictions of $\sour$ and $\tar$ to $J^{-1}(0)$. In order to
verify that $\tilde{\sour}$ is a surjective submersion
($\tilde{\tar}$ can be treated analogously), note that
\begin{equation}\label{eq:rankds}
\ker(\dd \tilde{\sour})_g = \ker (\dd \sour)_g \cap T_gJ^{-1}(0) =
\ker (\dd \sour)_g \cap \ker (\dd J)_g,\;\;\; g\in J^{-1}(0).
\end{equation}
Since  $M\subseteq J^{-1}(0)$, as a consequence of
\eqref{eq:Jcocycle}, and  $\tilde{\sour}|_M = \id$, we see that
$\dd\tilde{\sour}|_M$ is onto. It follows from \eqref{eq:rankds} and
$J_A=\dd J|_A$ that
\begin{equation}\label{eq:rankds2}
\ker(\dd \tilde{\sour}) |_M = A\cap \ker(\dd J)|_M = J_A^{-1}(0),
\end{equation}
and a dimension count shows that
\begin{equation}\label{eq:dimcount}
\mathrm{rank}(J_A^{-1}(0))=\dim(J^{-1}(0))- \dim(M).
\end{equation}

For $g\in \Grd$, let $r_g: \sour^{-1}(\tar(g))\to
\sour^{-1}(\sour(g))$ be the associated right-translation map.
Condition \eqref{eq:Jcocycle} implies that $(\dd J)_g(\dd
r_g)_{\tar(g)} a = (\dd J)_{\tar(g)}(a)$ for $a \in A_{\tar(g)}$,
which, along with \eqref{eq:rankds}, shows that, for $g\in
J^{-1}(0)$,
$$
\ker(\dd \tilde{\sour})_g = \ker (\dd \sour)_g \cap \ker (\dd J)_g =
(\dd r_g)_{\tar(g)} A \cap \ker (\dd J)_g = (\dd
r_g)_{\tar(g)}(J_A^{-1}(0)).
$$
It now follows from a  dimension count using \eqref{eq:dimcount}
that $\mathrm{rank}(\dd\tilde{\sour})_g = \dim(J^{-1}(0)) -
\mathrm{rank}(J_A^{-1}(0)) = \dim(M)$, so $\tilde{\sour}$ is a
submersion.

A direct consequence of \eqref{eq:Jcocycle} is that $J^{-1}(0)$ is
closed under the multiplication on $\Grd$,  which endows $J^{-1}(0)$
with a Lie groupoid structure with source and target maps given by
$\tilde{\sour}$ and $\tilde{\tar}$. As a result of
\eqref{eq:rankds2}, its Lie algebroid agrees with $J_A^{-1}(0)$.
\end{proof}

As in Section \ref{sec:infreduction}, we henceforth assume that the
$G$-action on $M$ is free and proper, which implies that the lifted
$G$-action on $\Grd$ is also free and proper \cite[Prop.~4.4]{FOR}.
If $0$ is a clean value for $J$, so that $J^{-1}(0)$ is a
submanifold (necessarily $G$-invariant by the equivariance of $J$),
then the quotient
$$
\Grd_{red}':= J^{-1}(0)/G
$$
is a smooth manifold that naturally inherits a closed 2-form
$\omega_{red}'$; indeed, just as in usual Marsden-Weinstein
reduction, condition \eqref{eq:mommap} guarantees that the pull-back
of $\omega$ to $J^{-1}(0)$ is basic with respect to the $G$-action.
Let $A_{red}$ and $\mu_{red}$ be as in Prop.~\ref{prop:Ared}.

\begin{proposition}\label{prop:globalred}
Suppose $0\in \frakg^*$ is a clean value for $J$. Then
\begin{enumerate}
\item $\Grd_{red}'$ is a Lie groupoid over $M/G$ such that the quotient map $J^{-1}(0)\to
\Grd_{red}'$ is a groupoid homomorphism, and whose Lie algebroid is
$A_{red}$.

\item The closed 2-form $\omega_{red}'$ on $\Grd_{red}'$ is
multiplicative and integrates the closed IM 2-form $\mu_{red}$ on
$A_{red}$.
\end{enumerate}
\end{proposition}

\begin{proof}
The fact that $\Grd_{red}'$ inherits a groupoid structure, uniquely
determined by the property that the quotient map
$$
p: J^{-1}(0)\to \Grd_{red}'
$$
is a homomorphism, can be directly checked using the freeness of the
$G$-action (note that the restriction of this map to identity
section $M\subseteq J^{-1}(0)$ agrees with the quotient map
\eqref{eq:quotmap}, hence the abuse of notation). Let
$A(\Grd_{red}')$ be the Lie algebroid of $\Grd_{red}'$. The map $p$
induces a surjective morphism of Lie algebroids
\begin{equation}\label{eq:dq}
\dd p |_{J_A^{-1}(0)}: J_A^{-1}(0)\to A(\Grd_{red}'),
\end{equation}
covering the quotient map $M\to M/G$. We observe that $\ker(\dd
p)\cap \ker(\dd \tilde{\sour})=\{0\}$ since $X\in \ker(\dd p)$ if
and only if $X=u_\Grd$ for some $u\in \frakg$, and
$\dd\tilde{\sour}(u_\Grd)=u_M=0$ if and only if $u=0$ by freeness.
It follows that the kernel of \eqref{eq:dq} is trivial, so that
\eqref{eq:dq} induces an identification between $J_A^{-1}(0)$ and
the pull-back bundle $p^* A(\Grd_{red}')$, from where we obtain a
natural identification between $A(\Grd_{red}')$ and $J_A^{-1}(0)/G=
A_{red}$.

For the first assertion in (2), note that the pull-back
$\iota^*\omega$ with respect to the inclusion $\iota: J^{-1}(0)\to
\Grd'$ is a closed, multiplicative 2-form on $J^{-1}(0)$, and one
can check that $\omega_{red}'$ is multiplicative from the equality
$p^*\omega_{red}'=\iota^*\omega$ and the fact that $p$ is a groupoid
homomorphism.

To compare IM 2-forms, recall that $\mu_{red}: A_{red}\to T^*(M/G)$
is defined as the $G$-quotient of the restriction (see
Prop.~\ref{prop:Ared})
$$
\mu|_{J_A^{-1}(0)}: J_A^{-1}(0) \to J_{can}^{-1}(0)=K^\circ.
$$
Consider $\underline{a}\in A_{red}$ and $\underline{X}\in T(M/G)$ at
a point $p(x) \in M/G$, and let $a \in J_A^{-1}(0)|_x$, $X\in T_xM$
be such that $\dd p(X) = \underline{X}$ and $\dd p (a) =
\underline{a}$. Then
$$
\omega_{red}'(\underline{a},\underline{X}) =
p^*\omega_{red}(a,X)=\iota^*\omega(a, X)= \mu(a)(X) =
\mu_{red}(\underline{a})(\underline{X}),
$$
which concludes the proof.
\end{proof}

When the 2-form $\omega$ on $\Grd$ is nondegenerate, then
 condition \eqref{eq:mommap} and the freeness of the $G$-action on
$\Grd$ guarantee that $0$ is a regular value for $J$. In general,
however, $\omega$ may be degenerate and \eqref{eq:mommap} may give
no information about the regularity of $J^{-1}(0)$. There is,
nevertheless, an alternative route for constructing a ``reduced''
Lie groupoid over $M/G$ that always works.

Let $\Grd_0$ be the source-simply-connected Lie groupoid integrating
the Lie algebroid $J_A^{-1}(0)$ (since $A$ is integrable, so is any
Lie subalgebroid of $A$). Since $J_A^{-1}(0)$ is a Lie subalgebroid
of $A$, $\Grd_0$ comes equipped with a groupoid homomorphism
$$
\iota_0: \Grd_0 \to \Grd,
$$
(which is an immersion but may fail to be injective, see e.g.
\cite{McMo}). The $G$-action on $J_A^{-1}(0)$ integrates to a
$G$-action on $\Grd_0$ for which $\iota_0$ is $G$-equivariant;
moreover, since the $G$-action on $M$ is free and proper, so is the
action on $\Grd_0$. The map $J\circ \iota_0: \Grd_0\to \frakg^*$ is
a multiplicative function whose infinitesimal counterpart $\dd
(J\circ \iota_0)|_{J_A^{-1}(0)}=J_A|_{J_A^{-1}(0)}$ vanishes. So
$J\circ \iota_0 =0$, i.e., $\iota_0(\Grd_0)\subseteq J^{-1}(0)$. As
a result, one may use \eqref{eq:mommap} to directly verify that
$\iota_0^*\omega$ is a closed multiplicative 2-form on $\Grd_0$
which is basic with respect to the $G$-action. Hence
$$
\Grd_{red} := \Grd_0/G
$$
is a smooth manifold, and it inherits a closed 2-form $\omega_{red}$
from $\iota_0^*\omega$.

\begin{proposition}\label{prop:globalred2}
The following holds:
\begin{enumerate}
\item $\Grd_{red}$ is a Lie groupoid over $M/G$ for which the quotient map $\Grd_0 \to
\Grd_{red}$ is a groupoid homomorphism, and whose Lie algebroid is
$A_{red}$.

\item The closed 2-form $\omega_{red}$ on $\Grd_{red}$ is
multiplicative, and it integrates the closed IM 2-form $\mu_{red}$
on $A_{red}$.
\end{enumerate}
\end{proposition}

The proof of Proposition~\ref{prop:globalred2} is similar to that of
Proposition~\ref{prop:globalred}, but with no regularity assumptions
on $J$.

As discussed in \cite{FOR}, the reduced groupoids $\Grd_{red}$ or
$\Grd_{red}'$ generally do not agree with the
source-simply-connected integration of $A_{red}$;
\cite[Thm.~4.11]{FOR} shows how one can approach the problem of
finding obstructions.

\section{The case of Dirac structures}\label{sec:dirac}

Let $M$ be a manifold equipped with a Dirac structure $L$ and a free
and proper action of a Lie group $G$ by Dirac diffeomorphisms. As
explained in Section \ref{subsec:dirac}, this fits into the set-up
in $(i)$, $(ii)$, $(iii)$ as follows: $A$ is the natural Lie
algebroid defined on the vector-bundle $L\to M$ by \eqref{eq:LAstr}
and $\mu = \mu_L$ is the closed IM 2-form defined in
\eqref{eq:diracIM}. As in Section \ref{sec:infreduction}, let
$K\subseteq TM$ be the distribution tangent to the orbits of $G$ on
$M$, and let $K^\perp = TM\oplus K^\circ$ denote its orthogonal in
$TM\oplus T^*M$ relative to the canonical pairing \eqref{eq:pair}.
Note that
\begin{equation}\label{eq:JA2}
L\cap K^\perp = \{ (X,\alpha) \in L\;|\; \alpha\in K^\circ\}=
\mu^{-1}(K^\circ)=J_A^{-1}(0).
\end{equation}
Assuming that $L\cap K^\perp$ has constant rank, following
Prop.~\ref{prop:Ared}, one may reduce $A$ and $\mu$ to obtain
$A_{red}$ and $\mu_{red}$, where
\begin{equation}\label{eq:Ared}
A_{red} = J_A^{-1}(0)/G = L\cap K^\perp/G
\end{equation}
is a Lie algebroid over $M/G$ and $\mu_{red}$ is a closed IM 2-form
on $A_{red}$. On the other hand, one may consider the reduction of
$L$ in the sense of Dirac structures, recalled below, which produces
a new Dirac structure on $M/G$. We will now compare these two
possible reduction procedures.

\subsection{Comparing reductions}

The lifted $G$-action on $\mathbb{T}M$ (see \eqref{eq:lift})
preserves the bundles $K$, $K^\perp$, and $L$, and we have a natural
identification
\begin{equation}\label{eq:ident}
\frac{K^\perp}{K}\Big / G  = \frac{TM}{K}\oplus K^\circ \Big/ G\cong
T(M/G) \oplus T^*(M/G).
\end{equation}
Assuming that the intersection $L\cap K^\perp$ has constant
rank\footnote{This is equivalent to $L\cap K$ having constant
rank.}, the quotient
\begin{equation}\label{eq:Lq}
L_{quot}:= \frac{L\cap K^{\perp} + K}{K}\Big /G \subset
\frac{K^\perp}{K}\Big /G
\end{equation}
defines a Dirac structure on $M/G$, see \cite[Sec.~4]{BCG}.
Equivalently, the Dirac structure \eqref{eq:Lq} is the result of a
push-forward by the quotient map $p:M\to M/G$:
\begin{equation}\label{eq:pushf}
(L_{quot})_{p(x)} = \{(\dd p(X), \beta) \,|\, X\in T_xM,\, \beta\in
T^*_{p(x)}(M/G),\, (X,\dd p^*\beta) \in L_x \}.
\end{equation}
Being a Dirac structure, $L_{quot}$ carries a Lie-algebroid
structure over $M/G$ and is naturally equipped with the closed IM
2-form (c.f \eqref{eq:diracIM})
\begin{equation}\label{eq:muquot}
\mu_{quot} := \pr_{T^*}|_{L_{quot}} : L_{quot}\to T^*(M/G).
\end{equation}
A direct comparison of \eqref{eq:Ared} and \eqref{eq:Lq} indicates
that $A_{red}$ and $L_{quot}$ generally disagree. The following
simple example illustrates this fact.

\begin{example}
Consider the Dirac structure $L=TM$ and $\mu=\mu_L = 0$. One
directly verifies that
$$
A_{red}=TM/G,\;\;\; \mu_{red}=0,
$$
while $L_{quot}=T(M/G)$ and $\mu_{quot}=0$, so the reductions do not
match. In this case, $\rho_{red}$ is the natural projection $TM/G
\to T(M/G)$, so $\rho_{red}(A_{red})=L_{quot}$.
\end{example}

We assume henceforth that $L\cap K^\perp$ has constant rank, so that
both $A_{red}$ and $L_{quot}$ are well-defined. Let us consider the
map
\begin{equation}\label{eq:redmap}
r: A_{red}\to T(M/G)\oplus T^*(M/G),\;\; a \mapsto
(\rho_{red}(a),\mu_{red}(a)).
\end{equation}
The following result relates $A_{red}$, $\mu_{red}$ and $L_{quot}$,
$\mu_{quot}$ in general.
\begin{lemma}\label{lem:2red}
The image of the map $r$ is $L_{quot}$ and its kernel is $(K\cap L)
/ G$. Moreover, $r: A_{red}\to L_{quot}$ is a Lie-algebroid morphism
and $\mu_{quot}\circ r = \mu_{red}$.
\end{lemma}

\begin{proof}
From \eqref{eq:Ared}, we have
$$
A_{red} = (K^\perp \cap L)/G = ((TM\oplus K^\circ)\cap L)/G.
$$
Let us consider the natural vector-bundle maps
$$
TM \to TM/G,\; X\mapsto \underline{X},\; \;\;\; K^\circ \to
K^\circ/G \cong T^*(M/G),\; \alpha \mapsto \underline{\alpha},
$$
as well as $\dd p: TM \to T(M/G)$, all covering the quotient map $p:
M\to M/G$. For $X \in T_xM$ and $\alpha \in (K^\circ)_x$, we have
$\underline{\alpha}(\dd p (X)) = \alpha(X)$, i.e., $(\dd p)^*
\underline{\alpha} = \alpha$. We can write
\begin{align*}
(A_{red})_{p(x)} & = \{ (\underline{X},\underline{\alpha})\,|\, X\in
T_xM, \alpha \in (K^\circ)_x, \, (X,\alpha)\in L_x \}\\
& = \{ (\underline{X},\underline{\alpha})\,|\, X\in T_xM, \, (X,(\dd
p)^* \underline{\alpha})\in L_x \}.
\end{align*}
With respect to this description of $A_{red}$, we have (c.f.
\eqref{eq:rhored})
\begin{equation}\label{eq:rhomu}
\rho_{red}(\underline{X},\underline{\alpha}) = \dd p (X), \;\;\;
\mu_{red}(\underline{X},\underline{\alpha}) = \underline{\alpha}.
\end{equation}
Using \eqref{eq:pushf}, we see that the image of $(A_{red})_{p(x)}$
under \eqref{eq:redmap} is
$$
\{ (\dd p(X), \underline{\alpha} ) \;|\; X\in T_xM, \, (X,(\dd p)^*
\underline{\alpha})\in L_x  \} = (L_{quot})_{p(x)}.
$$

Recalling that $K=\ker(\dd p)$, it is clear from \eqref{eq:rhomu}
that
$$
\ker (r) = \ker (\rho_{red})\cap \ker (\mu_{red}) = (K \cap L)/G.
$$
The fact that $r: A_{red}\to L_{quot}$ preserves Lie-algebroid
structures follows from $\rho_{red}$ preserving Lie brackets as well
as conditions \eqref{eq:IM1}, \eqref{eq:IM2} for $\mu_{red}$. The
compatibility $\mu_{quot}\circ r =\mu_{red}$ is immediate from
\eqref{eq:muquot}.
\end{proof}

The next result follows directly from Lemma~\ref{lem:2red}.

\begin{theorem}\label{thm:red}
The map $r$ defines a Lie-algebroid isomorphism from $A_{red}$ to
$L_{quot}$, such that $\mu_{quot}\circ r = \mu_{red}$,
 if and only if $K\cap L=\{0\}$.
\end{theorem}

We have the following consequences of Thm.~\ref{thm:red}:

\begin{enumerate}

\item Whenever $A_{red}$ and $L_{quot}$ do not coincide, i.e., $r$
fails to be an isomorphism, the Lie groupoid
$(\Grd_{red},\omega_{red})$ (resp. $(\Grd_{red}',\omega_{red}')$,
provided $J^{-1}(0)$ is a smooth submanifold) constructed in
Section~\ref{sec:globalred} is \textit{not} a presymplectic groupoid
for $L_{quot}$; by Thm.~\ref{thm:red}, this happens if and only if
$K\cap L=\{0\}$. In general, $(\Grd_{red},\omega_{red})$ (resp.
$(\Grd_{red}',\omega_{red}')$) is an \textit{over-presymplectic
groupoid} in the sense of \cite[Def.~4.5]{bcwz}; in particular, the
target map $\tar: (\Grd_{red},\omega_{red}) \to (M/G,L_{quot})$ is
always a forward Dirac map.

\item If $L$ is the graph of a Poisson structure $\pi$ on $M$, then
it immediately follows that $K\cap L=\{0\}$ and $L\cap K^\perp$ has
constant rank. In this case, $L_{quot}$ is just the graph of the
natural Poisson structure $\pi_{quot}$ on $M/G$, uniquely determined
by $p:M\to M/G$ being a Poisson map. By Thm.~\ref{thm:red},
$A_{red}$ is identified with $L_{quot}$ via $r$, and
$(\Grd_{red},\omega_{red})$ is a symplectic groupoid for
$\pi_{quot}$; since $0$ is automatically a regular value for $J$ in
this case, one may also consider the (possibly different) symplectic
groupoid  $(\Grd_{red}',\omega_{red}')$. This situation is fully
treated and further developed in \cite{FOR}.

\item Let us suppose that the Dirac structure $L$ on $M$ is such that
$L_{quot}$ is given by the graph of a Poisson structure
$\pi_{quot}$. In this case, we note that $r$ defines an isomorphism
$A_{red}\cong L_{quot}$ if and only if $L$ is itself the graph of a
Poisson structure, i.e., $L\cap TM=\{0\}$; indeed, it follows from
\eqref{eq:pushf} that
$$
L_{quot}\cap T(M/G) = \dd p (L\cap TM),
$$
so $L_{quot}\cap T(M/G)=\{0\}$ if and only if $L\cap TM \subseteq
K=\ker(\dd p)$. But since $L\cap K=\{0\}$, it follows that $L\cap
TM=\{0\}$.

As a consequence, the reduced groupoid $(\Grd_{red},\omega_{red})$
(resp. $(\Grd_{red}',\omega_{red}')$) is \textit{not} a symplectic
groupoid for $\pi_{quot}$ unless $L$ is the graph of a Poisson
structure to begin with. This observation indicates that the
reduction in \cite[Thm.~4.21]{FOR} for non-free Poisson actions may
not generally yield symplectic groupoids on each strata (just
\textit{over}-symplectic \cite[Def.~4.5]{bcwz}).
\end{enumerate}

\subsection{A non-integrable quotient}\label{subsec:nonint}

Let us keep considering a manifold $M$ equipped with a Dirac
structure $L$ and carrying a free and proper $G$-action by Dirac
diffeomorphisms. We saw in Thm.~\ref{thm:red} that $A_{red}$ and
$L_{quot}$ do not coincide in general. We now illustrate how
different these Lie algebroids can be concerning integrability.

It is a direct consequence of Prop.~\ref{prop:globalred2} that, if
$L$ is integrable as a Lie algebroid, then so is $A_{red}$. It may
happen, however, that the quotient Dirac structure $L_{quot}$ is
\textit{not} integrable as a Lie algebroid. Note that this feature
is not present when $L$ is defined by a Poisson structure, as in
this case $L_{quot}$ necessarily coincides with $A_{red}$; i.e, the
quotient of an integrable Poisson structure by a free and proper
$G$-action by Poisson diffeomorphisms is always integrable (see
\cite[Prop.~4.6]{FOR}).  We will verify how this picture changes in
the realm of Dirac structures through a concrete example.

Recall that $P= S^2 \times \mathbb{R}$, with coordinates $(x,t)$,
may be equipped with a Poisson structure $\pi_P$ that does not admit
an integration to a symplectic groupoid \cite{We87} (see also
\cite{CF2}). This nonintegrable Poisson structure is characterized
by the fact that its symplectic leaves are the spheres obtained by
constant values of $t$ endowed with $(1+t^2)$ times the usual
symplectic structure of $S^2$.

\begin{proposition}\label{prop:nonint}
There exists a presymplectic manifold $(M,\omega)$ carrying a free
and proper $G$-action preserving $\omega$ for which $P=M/G$ and the
quotient map $M\to P$ is a forward Dirac map; i.e., if $L$ is the
graph of $\omega$, then $L_{quot}$ is the graph of $\pi_P$.
\end{proposition}

Observe that whenever $L$ is the graph of a presymplectic structure,
it is automatically integrable as a Lie algebroid, as the projection
$\pr_T|_L : L\to TM$ is a Lie-algebroid isomorphism, so $L$ is
integrated by the fundamental groupoid of $M$, see e.g. \cite{CDW}.
As already mentioned, since $\pi_P$ is nonintegrable, it is
impossible to replace $\omega$ in Prop.~\ref{prop:nonint} by a
symplectic form (or any other integrable Poisson structure), so
$\omega$ must be degenerate.

The proof of Prop.~\ref{prop:nonint} is a direct consequence of the
construction in \cite[Example~6.8]{bcwz}, that we recall for
completeness.

Let $\frakg$ be a Lie algebra endowed with an $\Ad$-invariant,
symmetric and nondegenerate bilinear form $B: \frakg\times \frakg
\to \mathbb{R}$. We use $B$ to identify $\frakg^* \cong \frakg$, in
such a way that the adjoint and co-adjoint actions are intertwined.
We consider the function
$$
C: \frakg^* \to \mathbb{R}, \;\; \xi \mapsto \frac{1}{2}B(\xi,\xi).
$$
This function is a Casimir for the linear Poisson structure
$\pi_{\frakg^*}$ on $\frakg^*$, as a result of the $\Ad$-invariance
of $B$. It follows that $C^{-1}(\lambda)$ is a Poisson submanifold
of $\frakg^*$ for any $\lambda\in \mathbb{R}$, i.e.,
$\pi_{\frakg^*}$ restricts to a Poisson structure $\pi_\lambda$ on
$C^{-1}(\lambda)$ for which the inclusion
$$
j_\lambda : C^{-1}(\lambda) \to \frakg^*
$$
is a Poisson map.

Let $G$ be a connected Lie group whose Lie algebra is $\frakg$, and
let $G\times \frakg^*$ be endowed with the symplectic structure
$\Omega$ coming from the identification $G\times \frakg^* \cong
T^*G$ via right translations, so that the projection on the second
factor,
$$
p: G\times \frakg^* \to \frakg^*,
$$
is a Poisson map. The submanifold
$M_\lambda = (C\circ p)^{-1}(\lambda) = G\times C^{-1}(\lambda)$ of
$G\times \frakg^*$ is equipped with a closed 2-form
$\omega_\lambda$, given by the pull-back of $\Omega$ by the
inclusion
$$
\iota_\lambda: M_\lambda \hookrightarrow G\times \frakg^*,
$$
and carries a free and proper $G$-action (by right multiplication on
the first factor) so that $M_\lambda/G = C^{-1}(\lambda)$. We denote
the quotient map by $ p_\lambda : M_\lambda \to C^{-1}(\lambda)$, so
that
$$
j_\lambda\circ p_\lambda = p\circ \iota_\lambda.
$$

Let $L$ be the Dirac structure on $M_\lambda$ given by the graph of
$\omega_\lambda$. Consider the distribution $K = \ker(p_\lambda)=TG
\subseteq TM_{\lambda}$. Since $M_{\lambda}\subset G\times
\frakg^*$, we may view $K$ in $T(G\times \frakg^*)$ and consider its
symplectic orthogonal $K^\Omega \subseteq T(G\times \frakg^*)$.

\begin{lemma}\label{lem:Komega}
$K^\Omega\subseteq TM_{\lambda}$ and $L\cap K^\perp =
\{(X,i_X\omega_{\lambda})\,|\, X\in K^\Omega\}$. In particular,
$L\cap K^\perp$ has constant rank.
\end{lemma}

\begin{proof}
The distribution $K^\Omega$ is spanned by hamiltonian vector fields
$X_{p^*f}$ of functions of the form $p^*f$, for $f\in
C^\infty(\frakg^*)$. Since $p$ is a Poisson map, we have $\dd p
(X_{p^*f}) = X_f$. It follows that
$$
\dd C (\dd p (X_{p^*f}))= \dd C (X_f)=0,
$$
for $C$ is a Casimir. Hence $K^\Omega \subseteq \ker(\dd (C\circ
p))=TM_{\lambda}$.

Since $K^\perp=TM_{\lambda}\oplus K^\circ$,  $L\cap K^\perp =
\{(X,i_X\omega_\lambda)\,|\, X\in TM_{\lambda},\,
i_X\omega_{\lambda}\in K^\circ\}$. But one can directly check that
$i_X\omega_{\lambda}\in K^\circ$ if and only if $X\in
TM_{\lambda}\cap K^\Omega = K^\Omega$.
\end{proof}

Let $L_{quot}$ be the reduction of $L$ as in \eqref{eq:Lq},
\eqref{eq:pushf}.

\begin{lemma}\label{lem:pf}
$L_{quot}$ is the graph of $\pi_\lambda$.
\end{lemma}

\begin{proof}
For each $\xi \in C^{-1}(\lambda)$ and $\sigma \in G$, we have
$$
(L_{quot})_\xi = \{ (\dd p_\lambda(X), (\dd f)_\xi)\,|\, X\in
T_{(\sigma,\xi)}M_\lambda,\, f\in C^\infty(C^{-1}(\lambda)),\,
\mbox{ and } \, i_X \omega_\lambda = \dd p_\lambda^*f \}.
$$
We must compare it with the graph of $\pi_\lambda$, given at $\xi\in
C^{-1}(\lambda)$ by
$$
\{ ((\pi_\lambda)^\sharp((\dd f)_\xi),(\dd f)_\xi)\,|\, f\in
C^\infty(C^{-1}(\lambda))\}.
$$
To conclude that they coincide, it is enough to check that
$L_{quot}$ is contained in the graph of $\pi_\lambda$, as both
vector-bundles have the same rank. In other words, it suffices to
show that for $X \in T_{(\sigma,\xi)}M_\lambda$ and $f\in
C^\infty(C^{-1}(\lambda))$ such that $i_X \omega_\lambda = \dd
p_\lambda^*f$ then
\begin{equation}\label{eq:tocheck}
\dd p_\lambda(X) = (\pi_\lambda)^\sharp((\dd f)_\xi).
\end{equation}

Let $\hat{f} \in C^\infty(\frakg^*)$ be any extension of $f$, so
that $f=\hat{f}\circ j_\lambda$. Then
$$
\iota_\lambda^* (i_X  \Omega) = i_X \iota_\lambda^*\Omega = i_X
\omega_\lambda = \dd p_\lambda^* j_\lambda^* \hat{f} =
\iota_\lambda^* \dd p^*\hat{f},
$$
which implies that
\begin{equation}\label{eq:piOmega}
i_X \Omega = \dd p^*\hat{f} + k \dd (p^*C)
\end{equation}
for some $k \in \mathbb{R}$. Denoting by $\pi_{\Omega}$ the Poisson
structure defined by $\Omega$, the fact that $p: G\times \frakg^*
\to \frakg^*$ is a Poisson map implies that
$$
\dd p (\pi_\Omega^\sharp \dd p^*C) = \pi_{\frakg^*}^\sharp (\dd
C)=0,
$$
since $C$ is a Casimir. It follows from \eqref{eq:piOmega} that
$X=\pi_\Omega^\sharp(\dd p^*\hat{f} + k \dd (p^*C))$, so
$$
\dd p (X) = \dd p ( \pi_\Omega^\sharp(\dd p^*\hat{f} + k \dd
(p^*C))) = \dd p ( \pi_\Omega^\sharp(\dd p^*\hat{f})) =
\pi_{\frakg^*}^\sharp(\dd \hat{f}).
$$
We conclude that \eqref{eq:tocheck} holds as a direct consequence of
$j_\lambda$ being a Poisson map and $X$ being tangent to
$M_\lambda$.
\end{proof}

\begin{remark}\label{rem:Ared}
It follows from Lemma~\ref{lem:Komega} that $L\cap K^\perp$ is
isomorphic to the distribution $K^\Omega\subseteq TM_{\lambda}$,
which can be identified with the action Lie algebroid $\frakg\ltimes
M_{\lambda}$ arising from the diagonal action of $G$ on
$M_{\lambda}$ by left multiplication on the first factor and the
coadjoint action on the second. By \eqref{eq:Ared}, one can check
that $A_{red}=K^\Omega/G$ is identified with the action Lie
algebroid $\frakg\ltimes C^{-1}(\lambda)$  (relative to the
coadjoint action). If $G$ is simply connected, the reduced groupoid
$\Grd_{red}$ agrees with the action Lie groupoid $G\ltimes
C^{-1}(\lambda)$, and $\omega_{red}$ coincides with
$\omega_{\lambda}$.
\end{remark}

For the proof of Proposition~\ref{prop:nonint}, consider the Lie
algebras $\mathfrak{su}(2)$ and $\mathbb{R}$, equipped with their
canonical bilinear forms $B_{\mathfrak{su}(2)}$ and
$B_{\mathbb{R}}$. With the usual identification
$\mathfrak{su}(2)\cong \mathbb{R}^3$, $B_{\mathfrak{su}(2)}$ agrees
with the euclidean inner product. We let $\frakg$ be the Lie-algebra
direct sum $\mathfrak{su}(2) \oplus \mathbb{R}$, equipped with the
bilinear form $B = B_{\mathfrak{su}(2)} - B_{\mathbb{R}}$. Then
$$
C^{-1}(1/2)=\{ (\xi,t) \in \mathbb{R}^3 \times \mathbb{R}\,|\,
\SP{\xi,\xi} = 1+ t^2 \}
$$
can be identified with $P= S^2\times \mathbb{R}$ in such a way that
$\pi_{1/2}$ agrees with $\pi_P$. Lemma~\ref{lem:pf} shows that
$M=\mathrm{SU}(2) \times \mathbb{R} \times C^{-1}(1/2)$ can be
equipped with a presymplectic form that pushes forward to $\pi_P$.
This concludes the proof of Proposition~\ref{prop:nonint}.

\section{Relation with the path-space model}\label{sec:pathspaces}

Given an integrable Lie algebroid $A$ over $M$, there is a canonical
model for the source-simply-connected Lie groupoid integrating it in
terms of paths in $A$ \cite{CF} (see also \cite{catfel,severa}): one
considers the space $\widetilde{P}(A)$ of all $C^1$ paths $a: I\to
A$ from the interval $I=[0,1]$ into $A$ so that the projected path
$q_A\circ a:I\to M$, where $q_A:A \to M$ is the vector-bundle
projection, is of class $C^2$. This space has the structure of a
Banach manifold. We let $P(A)$ be the submanifold of
\textit{$A$-paths}, equipped with the equivalence relation given by
\textit{$A$-homotopy}, denoted by $\sim$, see \cite{CF}. Then the
quotient
$$
\Grd(A) := P(A)/\sim
$$
is a source-simply-connected Lie groupoid integrating $A$. This
explicit model was used in the study of Poisson actions in
\cite{FOR}. We now briefly explain how it relates to our approach
via closed IM-forms.

Let us consider the set-up described in $(i)$, $(ii)$, $(iii)$ of
Section ~\ref{subsec:setup}, let $J_A: A\to T^*M$ be as in
\eqref{eq:JA} and $J:\Grd(A)\to \frakg^*$ be its global counterpart,
see Prop.~\ref{prop:lifted}.

\begin{proposition}\label{prop:formula}
The map $J:\Grd(A)\to \frakg^*$ is given by
\begin{equation}\label{eq:Jformula}
\SP{J([a]),u} = \int_I \SP{\mu(a(t)),u_M|_{q_A(a(t))}} dt,
\end{equation}
where $u\in \frakg$, $u_M$ is the associated infinitesimal generator
and $a \in P(A)$.
\end{proposition}

When $A$ is defined by a Dirac structure $L\subset TM\oplus T^*M$,
one replaces $\mu$ by $\pr_{T^*}$ in the formula
\eqref{eq:Jformula}; if $L$ is the graph of a Poisson structure, one
recovers the formula for $J$ in \cite[Thm.~3.3]{FOR}.

Proposition~\ref{prop:formula} follows from a more general
observation. Let $A_1\to M_1$ and $A_2\to M_2$ be Lie algebroids and
$\psi: A_1\to A_2$ be a vector-bundle map. We denote by
$\widetilde{\psi}: \widetilde{P}(A_1)\to \widetilde{P}(A_2)$ the
induced map on paths.

\begin{lemma}\label{lem:compat}
$\widetilde{\psi}$ takes $A_1$-paths to $A_2$-paths (i.e.,
$\widetilde{\psi}(P(A_1))\subseteq P(A_2)$) preserving $A$-homotopy
(i.e., $a \sim_{A_1} a'$ implies that $\widetilde{\psi}(a)\sim_{A_2}
\widetilde{\psi}(a')$ for all $a$, $a' \in P(A_1)$) if and only if
$\psi$ is a Lie-algebroid morphism.
\end{lemma}

If $\psi$ is a Lie-algebroid morphism, it follows that the map
$\widetilde{\psi}|_{P(A_1)}: {P}(A_1)\to {P}(A_2)$ descends to a map
$\Grd(A_1)\to \Grd(A_2)$, which is the groupoid morphism integrating
$\psi$.

\begin{proof}(of Prop.~\ref{prop:formula})
For a vector space $V$, thought of as a trivial Lie algebra (or a
trivial Lie algebroid over a point), $\widetilde{P}(V)=P(V)$ and the
quotient map $P(V)\to \Grd(V)=V$ is given by $a(t)\mapsto \int_I
a(t) dt$.

Considering the Lie-algebroid morphism $J_A = J_{can}\circ \mu : A
\to \frakg^*$, it follows that the composition of
$\widetilde{J}_A|_{P(A)}: {P}(A) \to {P}(\frakg^*)$ with
$P(\frakg^*)\to \Grd(\frakg^*)=\frakg^*$ is
$$
a(t) \mapsto \int_I J_A(a(t)) dt.
$$
By Lemma~\ref{lem:compat}, the map $J: \Grd(A)\to
\Grd(\frakg^*)=\frakg^*$ is given by $J([a(t)]) = \int_I J_A(a(t))
dt$, hence $\SP{J([a(t)]),u} = \int_I \SP{J_A(a(t)),u} dt = \int_I
\SP{\mu(a(t)),u_M|_{q_A(a(t))}}.$
\end{proof}

One can also use Lemma~\ref{lem:compat} and \cite{BuCa} to
generalize formula \eqref{eq:Jformula} to describe multiplicative
$k$-forms  in terms of the path-space model.


\end{document}